\documentclass[oneside,english]{amsart}
\usepackage[T1]{fontenc}
\usepackage{amsthm}
\usepackage{amssymb}
\usepackage{esint}
\usepackage[all]{xy}
\usepackage{lmodern}

\makeatletter

\numberwithin{equation}{section}
\numberwithin{figure}{section}
\theoremstyle{plain}
\newtheorem{thm}{\protect\theoremname}[section]
  \theoremstyle{plain}
  \newtheorem{prop}[thm]{\protect\propositionname}
  \theoremstyle{definition}
  \newtheorem{defn}[thm]{\protect\definitionname}
  \theoremstyle{remark}
  \newtheorem{rem}[thm]{\protect\remarkname}
  \theoremstyle{plain}
  \newtheorem{lem}[thm]{\protect\lemmaname}
  \theoremstyle{definition}
  \newtheorem{example}[thm]{\protect\examplename}
   
  \theoremstyle{definition}

\makeatother

\usepackage{babel}
  \providecommand{\definitionname}{Definition}
  \providecommand{\examplename}{Example}
  \providecommand{\lemmaname}{Lemma}
  \providecommand{\propositionname}{Proposition}
  \providecommand{\remarkname}{Remark}
  \providecommand{\corollaryname}{Corollary}
\providecommand{\theoremname}{Theorem}

\DeclareMathOperator*{\bigtimes}{\times}

\begin{document}

\title{Classifying spaces of algebras over a prop}

\address{Laboratoire Paul Painlevé, Université de Lille 1, Cité Scientifique,
59655 Villeneuve d'Ascq Cedex, France}
\email{Sinan.Yalin@ed.univ-lille1.fr}

\author{Sinan Yalin}
\begin{abstract}
We prove that a weak equivalence between cofibrant props induces a weak equivalence between
the associated classifying spaces of algebras. This statement generalizes to the
prop setting a homotopy invariance result which is well known
in the case of algebras over operads. The absence of model category structure
on algebras over a prop leads us to introduce new methods to overcome this difficulty.
We also explain how our result can be extended to algebras over colored
props in any symmetric monoidal model category tensored over chain
complexes.


\textit{AMS} : 18G55 ; 18D50 ; 18D10 ; 55U10.

\end{abstract}
\maketitle

\section*{Introduction}

\bigskip{}

The notion of a prop has been introduced by MacLane in algebra \cite{MLa2}.
The name prop is actually an acronym for ``product and permutation''.
Briefly, a prop $P$ is a double sequence of objects $P(m,n)$ whose
elements represent operations with $m$ inputs and $n$ outputs.

Certain categories of algebras, like associative, Poisson or Lie algebras,
have a structure which is fully determined by operations with a single
output. These categories are associated to props $P$ of a certain
form, where operations in components $P(m,1)$ generate the prop.
Boardman and Vogt coined the name ``categories of operators of standard
form'' to refer to props of this particular form \cite{BV1}. Peter May introduced
the axioms of operads to deal with the components $P(m,1)$ which
define the core of such prop structures \cite{May2}. The work of these authors was
initially motivated by the theory of iterated loop spaces, in
topology (see \cite{BV2} and \cite{May2}). Operads have now proved to be a powerful device to handle a
variety of algebraic structures occurring in many branches of
mathematics.

However, if one wants to deal with bialgebras it becomes necessary
to use general props instead of operads. Important examples appeared in particular
in mathematical physics and string topology : the Frobenius bialgebras
(whose category is equivalent to the category of two-dimensional topological
quantum field theories), the topological conformal field theories
(which are algebras over the chain Segal prop), or the Lie bialgebras
introduced by Drinfeld in quantization theory are categories of bialgebras associated to props.

The purpose of this article is to set up a theory for the homotopical
study of algebras over a (possibly colored) prop. In a seminal series
of papers at the beginning of the 80's, Dwyer and Kan investigated
the simplicial localization of categories. They proved that the simplicial
localization gives a good device to capture secondary homology structures
usually defined in the framework of Quillen's model categories (see \cite{DK}).
An important homotopy invariant of a model category is its classifying space, defined as the nerve of its
subcategory of weak equivalences. The interest of such a classifying space has been
shown in the work of Dwyer and Kan \cite{DK}, who proved that this classifying space encodes information about
the homotopy types of the objects and their internal symmetries, i.e their
homotopy automorphisms.
They also proved that such a classifying space is homotopy invariant under Quillen
equivalences of model categories.

The algebras over an operad in a model category themselves form, under
suitable assumptions, a model category. A consequence of usual results about
model categories is that the classifying space of such a category is homotopy invariant up
to the weak homotopy type of the underlying operad. Unfortunately, there is no model category
structure on the algebras over a prop in general. We cannot handle our motivating examples of bialgebras
occurring in mathematical physics and string topology by using this approach, and we aim to overcome this difficulty.

The basic problem is to compare categories of algebras over a prop.
In order to bypass difficulties due to the absence of model structure on these algebras,
our overall strategy is to stay at the prop level as far as possible, and to use
factorization and lifting properties in the model category of props.
The structure of an algebra over a prop $P$ can be encoded by a prop
morphism $P \rightarrow End_A$, where $End_A$ is the endomorphism prop
associated to $A$. One can construct a version of endomorphism props modeling
$P$-algebra structures on diagrams. We can in particular use these diagrams endomorphisms props
to define path objects in the category of $P$-algebras. But we need an analogue of this
device for a variable $P$-algebra $A$, not a fixed object. The idea is to perform such
a construction on the abstract prop $P$ itself before moving to endomorphism props.
Combining this ``prop of $P$-diagrams'' construction with lifting and factorization techniques,
we endow the category of $P$-algebras with functorial path objects.

Consequently, the first main outcome of our study is the following homotopy invariance theorem.
Let $Ch_{\mathbb{K}}$ be the category of $\mathbb{Z}$-graded chain complexes
over a field $\mathbb{K}$ of characteristic zero. Let $(Ch_{\mathbb{K}})^P$ be the category of algebras associated to a prop $P$
in this category, and $w(Ch_{\mathbb{K}})^P$ its subcategory obtained by restriction to morphisms which are
weak equivalences in $Ch_{\mathbb{K}}$. Our result reads:

\medskip{}

\begin{thm}
Let $\varphi:P\stackrel{\sim}{\rightarrow}Q$
be a weak equivalence between two cofibrant props. The map $\varphi$
gives rise to a functor $\varphi^{*}: w(Ch_{\mathbb{K}})^Q\rightarrow w(Ch_{\mathbb{K}})^P$
which induces a weak equivalence of simplicial sets $\mathcal{N}\varphi^{*}:\mathcal{N} w(Ch_{\mathbb{K}})^Q\stackrel{\sim}{\rightarrow}\mathcal{N} w(Ch_{\mathbb{K}})^P$.
\end{thm}
\medskip{}
We can remove the hypothesis about the characteristic of $\mathbb{K}$ if
we suppose that $P$ is a prop with non-empty inputs or outputs (see
Definition 1.12 and Theorem 1.13). We explain in Section 2.7 how to extend
Theorem 0.1 to the case of a category tensored over $Ch_{\mathbb{K}}$. In
Section 3, we also briefly show that the proof of Theorem 0.1 extends readily
to the colored props context if we suppose that $\mathbb{K}$ has characteristic
zero (this hypothesis is needed to put a model category structure on colored props in $Ch_{\mathbb{K}}$,
see the work of Johnson and Yau\cite{JY}).
Recall that examples include cofibrant resolutions of the props encoding associative bialgebras, Lie bialgebras,
Frobenius bialgebras or Poisson bialgebras for instance. Algebras over a cofibrant resolution of a given
prop $P$ are called homotopy $P$-algebras. Theorem 0.1 implies that the classifying space
does not depend on the choice of the cofibrant resolution, and thus provides a well defined
homotopy invariant.

\begin{rem}
We do not address the case of simplicial sets. However, Theorem 1.4
in \cite{JY} endows the algebras over a colored prop in simplicial
sets with a model category structure. Moreover, the free algebra functor
exists in this case. Therefore one can transpose the methods used
in the operadic setting to obtain a simplicial version of Theorem
0.1. We also conjecture that our results
have a version in simplicial modules which follows from arguments
similar to ours.\end{rem}

\medskip{}

\textit{Organization}: the overall setting is reviewed in Section
1. We recall some definitions about symmetric monoidal categories
over a base category and axioms of monoidal model categories. Then
we introduce the precise definition of props and algebras over a prop.
We conclude these preliminaries with a fundamental result, the existence
of a model structure on the category of props.

The heart of this paper consists of Section 2, devoted to the proof of Theorem
0.1. The proof of Theorem 0.1 is quite long
and has been consequently divided in several steps.
Section 2.1 gives a sketch of our main arguments. In Sections 2.3, 2.4 and 2.5, we define particular props
called props of $P$-diagrams, which allow us to build a functorial path object in $P$-algebras.
In Section 2.6, we give a proof of Theorem 0.1.
At the end of Section 2 we generalize Theorem 0.1 to categories
tensored over $Ch_{\mathbb{K}}$.
Finally, we quickly present in Section 3 the extension of our arguments to colored props.

\medskip{}

\textbf{Acknowledgements.} I would like to thank my advisor Benoit Fresse for his useful and constant help.
I am grateful to Jim Stasheff and Eric Hoffbeck for their relevant remarks on the first version of this paper.
I also wish to thank David Chataur for stimulating discussions about these moduli spaces.
I thank the referee for his thorough report and detailed suggestions.

\medskip{}

\section{Recollections and general results}

\bigskip{}

\subsection{Symmetric monoidal categories over a base category}
\begin{defn}
Let $\mathcal{C}$ be a symmetric monoidal category. A \emph{symmetric monoidal
category over $\mathcal{C}$} is a symmetric monoidal category $(\mathcal{E},\otimes_{\mathcal{E}},1_{\mathcal{E}})$
endowed with a symmetric monoidal functor $\eta:\mathcal{C}\rightarrow\mathcal{E}$,
that is, an object under $\mathcal{C}$ in the $2$-category of symmetric monoidal categories.

This defines on $\mathcal{E}$ an external tensor product $\otimes :\mathcal{C}\times\mathcal{E}\rightarrow\mathcal{E}$
by $C\otimes X = \eta(C)\otimes_{\mathcal{E}} X$ for every $C\in\mathcal{C}$ and $X\in\mathcal{E}$.
This external tensor product is equipped with the following natural unit, associativity and symmetry isomorphisms:

(1) $\forall X\in\mathcal{E},1_{\mathcal{C}}\otimes X\cong X$,

(2) $\forall X\in\mathcal{E},\forall C,D\in\mathcal{C},(C\otimes D)\otimes X\cong C\otimes (D\otimes X)$,

(3) $\forall C\in\mathcal{C},\forall X,Y\in\mathcal{E},C\otimes (X\otimes Y)\cong(C\otimes X)\otimes Y\cong X\otimes(C\otimes Y)$.

The coherence constraints of these natural isomorphisms
(associativity pentagons, symmetry hexagons and unit triangles which mix both internal and external tensor products)
come from the symmetric monoidal structure of the functor $\eta$.

We will implicitly assume throughout the paper that all small limits and small
colimits exist in $\mathcal{C}$ and $\mathcal{E}$, and that each of these categories admit an internal hom bifunctor.
We suppose moreover the existence of an
external hom bifunctor $Hom_{\mathcal{E}}(-,-):\mathcal{E}^{op}\times\mathcal{E}\rightarrow\mathcal{C}$
satisfying an adjunction relation
\[
\forall C\in\mathcal{C},\forall X,Y\in\mathcal{E},Mor_{\mathcal{E}}(C\otimes X,Y)\cong Mor_{\mathcal{C}}(C,Hom_{\mathcal{E}}(X,Y))
\]
(so $\mathcal{E}$ is naturally an enriched category over $\mathcal{C}$).
\end{defn}
\medskip{}
\textbf{Examples.}

(1) The differential graded $\mathbb{K}$-modules (where $\mathbb{K}$ is a commutative ring) form a symmetric monoidal
category over the $\mathbb{K}$-modules. This is the main category used in this paper.

(2) Any symmetric monoidal category $\mathcal{C}$ forms
a symmetric monoidal category over $Set$ (the category of sets) with
an external tensor product defined by
\begin{align*}
\otimes :Set\times\mathcal{C} & \rightarrow\mathcal{C}\\
(S,C) & \mapsto\bigoplus_{s\in S}C.
\end{align*}

(3) Let $I$ be a small category; the $I$-diagrams in a symmetric monoidal category $\mathcal{C}$ form a symmetric monoidal
category over $\mathcal{C}$. The internal tensor product is defined pointwise, and the external tensor product
is defined by the functor $\eta$ which associates to $X\in\mathcal{C}$ the constant $I$-diagram $C_X$ on $X$.
The external hom $Hom_{\mathcal{C}^I}(-,-):\mathcal{C}^I\times\mathcal{C}^I\rightarrow\mathcal{C}$
is given by
\[
Hom_{\mathcal{C}^I}(X,Y)=\int_{i\in I}Hom_{\mathcal{C}}(X(i),Y(i)).
\]

\medskip{}

\begin{prop}
Let $F:\mathcal{D}\rightleftarrows\mathcal{E}:G$ be a symmetric monoidal adjunction
between two symmetric monoidal categories over $\mathcal{C}$. If $F$ preserves
the external tensor product then $F$ and $G$ satisfy an enriched
adjunction relation
\[
Hom_{\mathcal{E}}(F(X),Y)\cong Hom_{\mathcal{D}}(X,G(Y))
\]
 at the level of the external hom bifunctors (see \cite[Proposition 1.1.16]{Fre2}).
\end{prop}
\medskip{}

We now deal with symmetric monoidal categories equipped with a model
structure. We assume that the reader is familiar with the basics of model categories.
We refer to the paper of Dwyer and Spalinski \cite{DS} for a complete and accessible
introduction, and to Hirschhorn \cite{Hir1} and Hovey \cite{Hov} for a comprehensive
treatment. We just recall the axioms of symmetric monoidal model categories
formalizing the interplay between the tensor and the model structures.

\medskip{}

\begin{defn}
Let $\mathcal{C}$ be a category with small colimits and $F:\mathcal{A}\times\mathcal{B}\rightarrow\mathcal{C}$
a bifunctor. The \emph{pushout-product} of two morphisms $f:A\rightarrow B\in\mathcal{A}$
and $g:C\rightarrow D\in\mathcal{B}$ is the morphism
\[
(f_{*},g_{*}):F(A,D)\oplus_{F(A,C)}F(B,C)\rightarrow F(B,D)
\]
given by the commutative diagram:

\[
\xymatrix{F(A,C)\ar[d]_{F(A,g)}\ar[r]^{F(f,C)} & F(B,C)\ar[d]\ar@/^{1pc}/[rdd]^{(F(B,g)}\\
F(A,D)\ar@/_{1pc}/[drr]_{F(f,D)}\ar[r] & F(A,D)\oplus_{F(A,C)}F(B,C)\ar[dr]^{(f_{*},g_{*})}\\
 &  & F(B,D)
}.
\]

\end{defn}
\medskip{}

\begin{defn}
(1) A \emph{symmetric monoidal model category} is a symmetric monoidal category
$\mathcal{C}$ equipped with a model category structure such that
the following axiom holds:

\textbf{MM1.} The pushout-product $(i_{*},j_{*}):A\otimes D\oplus_{A\otimes C}B\otimes C\rightarrow B\otimes D$
of cofibrations $i:A\rightarrowtail B$ and $j:C\rightarrowtail D$ is a cofibration
which is also acyclic as soon as $i$ or $j$ is so.

(2) Suppose that $\mathcal{C}$ is a symmetric monoidal model category.
A symmetric monoidal category $\mathcal{E}$ over $\mathcal{C}$ is
a symmetric monoidal model category over $\mathcal{C}$ if the axiom MM1 holds for both the internal and external tensor
products of $\mathcal{E}$.
\end{defn}
\medskip{}
\textbf{Example:} the category $Ch_{\mathbb{K}}$ of chain complexes over
a field $\mathbb{K}$ is our main working example of symmetric monoidal model category.
The weak equivalences of $Ch_{\mathbb{K}}$ are the quasi-isomorphisms, that is, the morphisms of chain complexes
inducing isomorphisms of graded vector spaces at the homology level. The fibrations are the degreewise surjections
and the cofibrations the degreewise injections.

\medskip{}

\begin{lem}
In a symmetric monoidal model category $\mathcal{E}$ over $\mathcal{C}$
the axiom MM1 for the external tensor product is equivalent to the
following one:

\textbf{MM1'.} The morphism
\[
(i^{*},p_{*}):Hom_{\mathcal{E}}(B,X)\rightarrow Hom_{\mathcal{E}}(A,X)\times_{Hom_{\mathcal{E}}(A,Y)}Hom_{\mathcal{E}}(B,Y)
\]
 induced by a cofibration $i:A\rightarrowtail B$ and a fibration
$p:X\twoheadrightarrow Y$ is a fibration in $\mathcal{C}$ which
is also acyclic as soon as $i$ or $p$ is so (cf. Lemma 4.2.2 in \cite{Hov}).
\end{lem}
\medskip{}
One can use the internal hom bifunctor to see that the axiom MM1 for the internal tensor product is in the
same way equivalent to a ``dual'' axiom MM1'.

\medskip{}

\subsection{On $\Sigma$-bimodules, props, and algebras over a prop}

Let $\mathcal{C}$ be a symmetric monoidal category admitting all
small limits and small colimits, whose tensor product preserves colimits,
and which is endowed with an internal hom bifunctor. Let $\mathbb{B}$ be the
category having the pairs $(m,n)\in\mathbb{N}^2$ as objects
together with morphism sets such that:
\[
Mor_{\mathbb{B}}((m,n),(p,q))=\begin{cases}
\Sigma_{m}^{op}\times\Sigma_{n}, & \text{if $(p,q)=(m,n)$},\\
\emptyset & \text{otherwise}.
\end{cases}
\]
The $\Sigma$-biobjects in $\mathcal{C}$ are the $\mathbb{B}$-diagrams
in $\mathcal{C}$. So a $\Sigma$-biobject is a double sequence $\{M(m,n)\in\mathcal{C}\}_{(m,n)\in\mathbb{N}^2}$
where each $M(m,n)$ is equipped with a right action of $\Sigma_{m}$
and a left action of $\Sigma_{n}$ commuting with each other. Let $\mathbb{A}$
be the discrete category of pairs $(m,n)\in\mathbb{N}^2$. We have
an obvious forgetful functor $\phi^{*}:\mathcal{C}^{\mathbb{B}}\rightarrow\mathcal{C}^{\mathbb{A}}$.
This functor has a left adjoint $\phi_{!}:\mathcal{C}^{\mathbb{A}}\rightarrow\mathcal{C}^{\mathbb{B}}$
defined on objects by

\begin{align*}
\forall M\in\mathcal{C}^{\mathbb{A}},\forall(m,n)\in\mathbb{N}^2,\phi_{!}M(m,n) & =1_{\mathcal{C}}[\Sigma_{n}\times\Sigma_{m}^{op}]\otimes M(m,n)\\
 & \cong\bigoplus_{\Sigma_{n}\times\Sigma_{m}^{op}}M(m,n).
\end{align*}

\medskip{}

\begin{defn}
(1) Let $\mathcal{C}$ be a symmetric monoidal category. A \emph{prop} in
$\mathcal{C}$ is a symmetric monoidal category $P$, enriched over
$\mathcal{C}$, with $\mathbb{N}$ as object set and the tensor product
given by $m\otimes n=m+n$ on objects. Let us unwrap this definition.
Firstly we see that a prop is a $\Sigma$-biobject. Indeed, the group
$\Sigma_{m}$ acts on $m=1+...+1=1^{\otimes m}$ and the group $\Sigma_{n}^{op}$
acts on $n=1+...+1=1^{\otimes n}$ by permuting the variables at the
morphisms level. A prop is endowed with horizontal products
\[
\circ_{h}:P(m_1,n_1)\otimes P(m_2,n_2)\rightarrow P(m_1+m_2,n_1+n_2)
\]
which are defined by the tensor product of homomorphisms, since
$P(m_1\otimes m_2,n_1\otimes n_2)=P(m_1+m_2,n_1+n_2)$
by definition of the tensor product on objects.
It also admits vertical composition products
\[
\circ_{v}:P(k,n)\otimes P(m,k)\rightarrow P(m,n)
\]
corresponding to the composition of homomorphisms, and units $1\rightarrow P(n,n)$
corresponding to identity morphisms of the objects $n\in\mathbb{N}$ in $P$. These
operations satisfy relations coming from the axioms of symmetric monoidal
categories. We refer the reader to Enriquez and Etingof \cite{EE} for an explicit description
of props in the context of modules over a ring. We denote by $\mathcal{P}$ the category of props.

Another construction of props is given in \cite{JY}:
props are defined there as $\boxtimes_{h}$-monoids in the $\boxtimes_{v}$-monoids
of colored $\Sigma$-biobjects, where $\boxtimes_{h}$ and $\boxtimes_{v}$
denote respectively a horizontal composition product and a vertical
composition product.
\end{defn}
\medskip{}
Appendix A of \cite{Fre1} provides a construction of the free prop
on a $\Sigma$-biobject. The free prop functor is left adjoint to
the forgetful functor:
\[
F:\mathcal{C}^{\mathbb{B}}\rightleftarrows\mathcal{P}:U.
\]

\medskip{}

\begin{defn}
(1) To any object $X$ of $\mathcal{C}$ we can associate an \emph{endomorphism
prop} $End_{X}$ defined by
\[
End_{X}(m,n)=Hom_{\mathcal{C}}(X^{\otimes m},X^{\otimes n}).
\]
The actions of the symmetric groups are the permutations of the input
variables and of the output variables, the horizontal product is the
tensor product of homomorphisms and the vertical composition product
is the composition of homomorphisms. The units $1_{\mathcal{C}}\rightarrow Hom_{\mathcal{C}}(X^{\otimes n},X^{\otimes n})$
represent $id_{X^{\otimes n}}$.

(2) An \emph{algebra over a prop $P$}, or $P$-algebra, is an object $X\in\mathcal{C}$
equipped with a prop morphism $P\rightarrow End_{X}$.
\end{defn}
\medskip{}
We can also define a $P$-algebra in a symmetric monoidal category
over $\mathcal{C}$:

\medskip{}

\begin{defn}
Let $\mathcal{E}$ be a symmetric monoidal category over $\mathcal{C}$.

(1) The endomorphism prop of $X\in\mathcal{E}$ is given by $End_{X}(m,n)=Hom_{\mathcal{E}}(X^{\otimes m},X^{\otimes n})$
where $Hom_{\mathcal{E}}(-,-)$ is the external hom bifunctor of $\mathcal{E}$.

(2) Let $P$ be a prop in $\mathcal{C}$. A $P$-algebra in $\mathcal{E}$
is an object $X\in\mathcal{E}$ equipped with a prop morphism $P\rightarrow End_{X}$.
\end{defn}

\medskip{}

\textbf{Example:} we recall from \cite{Fre1} an explicit definition in the case of a diagram category over $\mathcal{E}$:
let $\{X_i\}_{i\in I}$ be a $I$-diagram in $\mathcal{E}$, then
\[
End_{\{X_i\}_{i\in I}}=\int_{i\in I}Hom_{\mathcal{E}}(X_i^{\otimes m},X_i^{\otimes n}).
\]
This end can equivalently be defined as a coreflexive equalizer
\[
\xymatrix{End_{\{X_i\}}(m,n)\ar[r] & \prod_{i\in I}Hom_{\mathcal{E}}(X_i^{\otimes m},X_i^{\otimes n})\ar@<1ex>[r]^-{d_0}\ar@<-1ex>[r]_-{d_1} & \prod_{u:i\rightarrow j\in mor(I)}Hom_{\mathcal{E}}(X_i^{\otimes m},X_j^{\otimes n})\ar@/^{2pc}/[l]^{s_0}}
\]
where $d_0$ is the product of the maps
\[
u_{*}:Hom_{\mathcal{E}}(X_i^{\otimes m},X_i^{\otimes n})\rightarrow Hom_{\mathcal{E}}(X_i^{\otimes m},X_j^{\otimes n})
\]
 induced by the morphisms $u:i\rightarrow j$ of $I$ and $d_1$
is the product of the maps
\[
u^{*}:Hom_{\mathcal{E}}(X_j^{\otimes m},X_j^{\otimes n})\rightarrow Hom_{\mathcal{E}}(X_i^{\otimes m},X_j^{\otimes n})
\]
The section $s_0$ is the projection on the factors associated to
the identities $id:i\rightarrow i$.

This construction is functorial in $I$: given a $J$-diagram $\{X_j\}_{j\in J}$,
every functor $\alpha:I\rightarrow J$ gives rise to a prop morphism
$\alpha^{*}:End_{\{X_j\}_{j\in J}}\rightarrow End_{\{X_{\alpha(i)}\}_{i\in I}}$.

\bigskip{}

\subsection{The semi-model category of props}

Suppose that $\mathcal{C}$ is a cofibrantly generated symmetric monoidal
model category. The category of $\Sigma$-biobjects $\mathcal{C}^{\mathbb{B}}$
is a diagram category over $\mathcal{C}$, so it inherits a cofibrantly
generated model category structure. The weak equivalences and fibrations are
defined componentwise. The generating (acyclic) cofibrations are given
by $i\otimes \phi_{!}G_{(m,n)}$ , where $(m,n)\in\mathbb{N}^2$
and $i$ ranges over the generating (acyclic) cofibrations of $\mathcal{C}$.
Here $\otimes$ is the external tensor product of $\mathcal{C}^{\mathbb{B}}$
and $G_{(m,n)}$ is the double sequence defined by
\[
G_{(m,n)}(p,q)=\begin{cases}
1_{\mathcal{C}}, & \text{if $(p,q)=(m,n)$},\\
0 & \text{otherwise}.
\end{cases}
\]

We can also see this result as a transfer of cofibrantly generated
model category structure via the adjunction $\phi_{!}:\mathcal{C}^{\mathbb{A}}\rightleftarrows\mathcal{C}^{\mathbb{B}}:\phi^{*}$
(via exactly the same proof as in the case of $\Sigma$-objects, see for instance Proposition 11.4.A in \cite{Fre2}).
The question is to know whether the adjunction $F:\mathcal{C}^{\mathbb{B}}\rightleftarrows\mathcal{P}:U$
transfer this model category structure to the props. In the general
case it works only with the subcategory $\mathcal{P}_0$ of props
with non-empty inputs or outputs and does not give a full model category
structure. We give the precise statement in Theorem 1.10.

\medskip{}

\begin{defn}
A $\Sigma$-biobject $M$ has \emph{non-empty inputs} if it satisfies
\[
M(0,n)=\begin{cases}
1_{\mathcal{C}}, & \text{if $n=0$},\\
0 & \text{otherwise}.
\end{cases}
\]
We define in a symmetric way a $\Sigma$-biobject with \emph{non-empty outputs}.
The category of $\Sigma$-biobjects with non-empty inputs is noted
$\mathcal{C}_0^{\mathbb{B}}$.
\end{defn}
\medskip{}

The composite adjunction $ $
\[
\mathcal{C}^{\mathbb{A}}\rightleftarrows\mathcal{C}^{\mathbb{B}}\rightleftarrows\mathcal{P}
\]
restricts to an adjunction
\[
\mathcal{C}_0^{\mathbb{A}}\rightleftarrows\mathcal{C}_0^{\mathbb{B}}\rightleftarrows\mathcal{P}_0.
\]
 We define the weak equivalences (respectively fibrations) in $\mathcal{P}_0$
componentwise, i.e their images by the forgetful functor $U:\mathcal{P}_0\rightarrow\mathcal{C}_0^{\mathbb{A}}$
are weak equivalences (respectively fibrations) in $\mathcal{C}_0^{\mathbb{A}}$.
We define the generating (acyclic) cofibrations as the images under the
free prop functor of the generating (acyclic) cofibrations of $\mathcal{C}_0^{\mathbb{B}}$.
We have the following result:

\medskip{}

\begin{thm}
(cf. \cite{Fre1}, Theorem 4.9) Let $\mathcal{C}$ be a cofibrantly
generated symmetric monoidal model category. Suppose moreover that the unit of $\mathcal{C}$ is cofibrant.
Then the category $\mathcal{P}_0$
of props with non-empty inputs (or outputs) equipped with the
classes of weak equivalences, fibrations and cofibrations of 1.3 forms a semi-model category.
Moreover the forgetful functor $U:\mathcal{P}_0\rightarrow\mathcal{C}_0^{\mathbb{A}}$
preserves cofibrations with cofibrant domain.
\end{thm}
\medskip{}
A semi-model category structure is a slightly weakened version of model category structure: the
lifting axioms hold only for cofibrations with cofibrant domain, and the factorization axioms hold only
on maps with cofibrant domain (see the relevant section of \cite{Fre1}).
The notion of a semi-model category is sufficient to do homotopy theory. In certain categories
we recover a full model structure on the whole category of props:

\medskip{}

\begin{thm}
(cf. \cite{Fre1}, Theorem 5.5) If the base category $\mathcal{C}$
is the category of dg-modules over a ring $\mathbb{K}$ such that $\mathbb{Q}\subset \mathbb{K}$,
simplicial modules over a ring, simplicial sets or topological spaces,
then there is a transfer of model category structure on the whole category of props via the adjunction
$\mathcal{C}^{\mathbb{A}}\rightleftarrows\mathcal{P}$.
\end{thm}
\bigskip{}

\section{Homotopy invariance of the classifying space}

\bigskip{}

The purpose of this section is to establish Theorem 0.1.
We give the details of our arguments
in the case $\mathcal{E}=\mathcal{C}=Ch_{\mathbb{K}}$ (the $\mathbb{Z}$-graded
chain complexes over a field $\mathbb{K}$ of characteristic zero). Afterwards, we briefly explain
the generalization of these arguments when $\mathcal{E}$ is a cofibrantly
generated symmetric monoidal model category over $Ch_{\mathbb{K}}$.

\bigskip{}

\subsection{Statement of the result and outline of the proof}

In the work of Dwyer-Kan \cite{DK}, the classifying space of a category $\mathcal{M}$
equipped with a subcategory of weak equivalences $w\mathcal{M}$
is the simplicial set $\mathcal{N}(w\mathcal{M})$, where $\mathcal{N}$
is the simplicial nerve functor.
This simplicial set satisfies the following crucial property:
\begin{thm}(Dwyer-Kan)
Let $M$ be a category, $W$ a class of morphisms of $M$ and $wM$ the subcategory of $M$ defined by
$ob(wM)=ob(M)$ and $mor(wM)=W$. Then one has a homotopy equivalence
\[
\mathcal{N}wM\sim \coprod_{[X]} \overline{W}LwM(X,X)
\]
where $\mathcal{N}$ is the simplicial nerve functor, $[X]$ ranges over the weak equivalence
classes of the objects of $M$, $\overline{W}$ is the simplicial classifying space (see \cite{May}), and $L(-)$ is the simplicial localization functor.
When $M$ is a model category, one has moreover
\[
\mathcal{N}wM\sim \coprod_{[X]} \overline{W}haut(X)
\]
where $haut(X)$ is the simplicial monoid of self weak equivalences on a fibrant-cofibrant resolution of $X$.
\end{thm}

In the case of $\mathcal{E}^P$
(the $P$-algebras in $\mathcal{E}$ for a prop $P$ defined in $\mathcal{C}$)
we use the name ``classifying space'' to refer to the simplicial
set $\mathcal{N}w(\mathcal{E}^{cf})^P$, where $w(\mathcal{E}^{cf})^P$
is the subcategory of $P$-algebra morphisms whose underlying morphisms
in $\mathcal{E}$ are weak equivalences between fibrant-cofibrant
objects.

\subsubsection{The operadic case}

In the operadic context, algebras over operads satisfy the following fundamental property:
a weak equivalence between two cofibrant operads induces a weak equivalence between their
associated classifying spaces of algebras. The proof of this result is done in three steps.
Firstly, one shows the existence of an adjunction between the two categories of algebras:
if $\phi:P\rightarrow Q$ is a morphism of operads, it induces an adjunction
\[
\phi_!:\mathcal{E}^P\leftrightarrows \mathcal{E}^Q:\phi^*
\]
where $\phi^*$ is given on each $Q$-algebra $Q\rightarrow End_X$ by the precomposition $P\stackrel{\phi}{\rightarrow}Q\rightarrow End_X$
and $\phi_!$ is obtained via a certain coequalizer for which we refer the reader to \cite{Fre2}.
Secondly one proves that if $\phi$ is a weak equivalence and $P$ and $Q$ are "well-behaved" operads,
then this adjunction actually forms a Quillen equivalence,
presented in full generality in \cite{Fre2}. This Quillen equivalence is precisely
obtained between semi-model categories of algebras over weakly equivalent $\Sigma$-cofibrant operads, in a cofibrantly generated
symmetric monoidal model category $\mathcal{E}$ over a base category $\mathcal{C}$(see Theorem 12.5.A. of \cite{Fre2}, proved in Chapter 16).
Finally, according to the results of Dwyer-Kan, a Quillen equivalence induces a weak homotopy equivalence of
the classifying spaces (actually it induces much more, that is, a Dwyer-Kan equivalence of the simplicial
localizations).

\subsubsection{The key statement}

Such a method fails in the prop setting: one does not know how to construct a left adjoint of the functor $\phi^*$.
And even if such an adjoint exists, there is no free algebra functor, and
a model structure does not exist on the category of algebras over a prop except in some particular cases
such as simplicial sets (see \cite{JY}). So the difficult part is to deal with this absence of
model structure to get a similar result for algebras over props. Therefore, our method is
entirely different from this one. The crux of our proof is given by the following statement:

\medskip{}

\begin{thm}
Let $P$ be a cofibrant prop. The mappings $\mathcal{N}\varphi^{*},\mathcal{N}\psi^{*}:\mathcal{N}w(\mathcal{E}^{cf})^P\rightrightarrows\mathcal{N}w(\mathcal{E}^{cf})^{P}$
associated to homotopic prop morphisms $\varphi,\psi:P\rightrightarrows P$
are homotopic in $sSet$.
\end{thm}
\medskip{}

Let us outline the main steps of the proof of Theorem 2.2 in the case $\mathcal{E}=\mathcal{C}=Ch_{\mathbb{K}}$.
The idea is to construct a zigzag of natural transformations $\varphi^{*}\stackrel{\sim}{\leftarrow}Z\stackrel{\sim}{\rightarrow}\psi^{*}$,
where $Z$ is a functorial path object in $Ch_{\mathbb{K}}^P$. For this, we need in particular to obtain such a functorial
path object. The existence of a path object for algebras over props is proved in \cite{Fre1} (Section 8), but this path object is not functorial.
The main point of our proof is to solve this functoriality problem by ``correcting'' in a certain sense the $P$-algebra structure on the path object,
and then following arguments similar to those of \cite{Fre1} but with functorial $P$-actions. We proceed as follows.
We use functional notations $\mathcal{Y}(X)$, $\mathcal{Z}(X)$ and $\mathcal{V}(X)$ to refer to diagrams functorially associated to
an object $X$ which, in our constructions, ranges within (some subcategory of) $Ch_{\mathbb{K}}$.
We first consider the functorial path object diagram associated to any $X$ in $Ch_{\mathbb{K}}$

\medskip{}

\[
\mathcal{Y}(X): \xymatrix{ &  & X\\
X\ar@/^{1pc}/[urr]^{=}\ar@/^{-1pc}/[drr]_{=}\ar@{>->}[]!R+<4pt,0pt>;[r]_-{s}^-{\sim} & Z(X)\ar@{->>}[ur]_{d_0}^{\sim}\ar@{->>}[dr]^{d_1}_{\sim}\\
 &  & X
}.
\]
and its subdiagram $\mathcal{Z}(X)=\{X_0\stackrel{\sim}{\twoheadleftarrow}Z(X)\stackrel{\sim}{\twoheadrightarrow}X_1\}$.
We prove that the natural $P$-action existing on the diagram

\medskip{}

\[
\mathcal{V}(X): \xymatrix{ &  & X\\
X\ar[urr]^{=}\ar[drr]_{=}\\
 &  & X
}
\]
extends to a natural $P$-action on $\mathcal{Y}(X)$. For
this, we consider ``props of $P$-diagrams'', which
are built by replacing all the operations $X^{\otimes m}\rightarrow X^{\otimes n}$
in the endomorphism prop of a given diagram by operations
of $P(m,n)$. We use notations $End_{\mathcal{Y}(P)}$, $End_{\mathcal{Z}(P)}$ and $End_{\mathcal{V}(P)}$ to refer to
these props of $P$-diagrams. We verify that these constructions give rise to
props acting naturally on the endomorphism prop of the associated diagram.
We use these props of $P$-diagrams to give a $P$-action
on the zigzag of endofunctors $Id\stackrel{\sim}{\twoheadleftarrow}Z\stackrel{\sim}{\twoheadrightarrow}Id$.
We check that we retrieve the action given by $\varphi$ and $\psi$ on the extremity of this zigzag.
We thus have a zigzag connecting $\varphi^{*}$ and $\psi^{*}$
and yielding the desired homotopy between $\mathcal{N}\varphi^{*}$
and $\mathcal{N}\psi^{*}$.

\subsubsection{The argument line of the proof}

Let us first define the notion of a functorial $P$-action via a prop of $P$-diagrams:

\textbf{About functorial $P$-actions on diagrams.} We consider a diagram $\mathcal{D}(X)$ depending on $X$, which will be one of the three aforementioned diagrams
($\mathcal{V}(X)$, $\mathcal{Z}(X)$ and $\mathcal{Y}(X)$) and correspond to a certain functor denoted by $\mathcal{D}(X):I\rightarrow Ch_{\mathbb{K}}$.
We  will associate to $End_{\mathcal{D}(X)}$ a prop of $P$-diagrams
$End_{\mathcal{D}(P)}$.
When there exists a $P$-algebra structure on $X$, that is, a prop morphism $P\rightarrow End_X$,
this prop is equipped by construction with a morphism of props $ev_X:End_{\mathcal{D}(P)}\rightarrow End_{\mathcal{D}(X)}$.
This evaluation morphism sends every $\phi\in End_{\mathcal{D}(P)}(m,n)$ on a collection of morphisms $\phi_X(i):\mathcal{D}(X)(i)^{\otimes m}\rightarrow \mathcal{D}(X)(i)^{\otimes n}$ satisfying the following commutative diagram for every morphism $u:i\rightarrow j$ of $I$:
\[
\xymatrix{
\mathcal{D}(X)(i)^{\otimes m} \ar[d]_{\mathcal{D}(X)(u)^{\otimes m}} \ar[r]^{\phi_X(i)} & \mathcal{D}(X)(i)^{\otimes n} \ar[d]^{\mathcal{D}(X)(u)^{\otimes n}} \\
\mathcal{D}(X)(j)^{\otimes m} \ar[r]_{\phi_X(j)} & \mathcal{D}(X)(j)^{\otimes n}
}
\]
where $\mathcal{D}(X)(u)$ is the morphism in the diagram $\mathcal{D}(X)$ induced by $u$.
By construction of $End_{\mathcal{D}(P)}$, the prop morphism $ev_X$ satisfies a functoriality property with respect to $X$.
If $f:X\rightarrow Y$ is a morphism of $P$-algebras,
then the following diagrams are commutative for every $\phi\in End_{\mathcal{D}(P)}(m,n)$ and every $i\in I$:
\[
\xymatrix{
\mathcal{D}(X)(i)^{\otimes m} \ar[d]_{(\mathcal{D}(f)(i))^{\otimes m}} \ar[r]^{\phi_X(i)} & \mathcal{D}(X)(i)^{\otimes n} \ar[d]^{(\mathcal{D}(f)(i))^{\otimes n}} \\
\mathcal{D}(Y)(j)^{\otimes m} \ar[r]_{\phi_Y(j)} & \mathcal{D}(Y)(j)^{\otimes n}
}
\]
The commutativity of these diagrams implies that every morphism of props $P\rightarrow End_{\mathcal{D}(P)}$
induces a functorial $P$-action $P\rightarrow End_{\mathcal{D}(X)}$, that is, a functor
\[
\mathcal{D}:Ch_{\mathbb{K}}^P\rightarrow Func(I,Ch_{\mathbb{K}}^P)=Func(I,Ch_{\mathbb{K}})^P
\]
fitting in a commutative square
\[
\xymatrix{
Ch_{\mathbb{K}}\ar[r]^-{\mathcal{D}} & Func(I,Ch_{\mathbb{K}}) \\
Ch_{\mathbb{K}}^P \ar[r]_-{\mathcal{D}} \ar[u] & Func(I,Ch_{\mathbb{K}})^P \ar[u]
}
\]
in which the vertical arrows are the forgetful functors.

Our argument line is divided into two steps.

\textbf{Step 1.}
For every $X\in Ch_{\mathbb{K}}^P$, we have $End_{\mathcal{V}(X)}\cong End_X$
so the morphism $P\rightarrow End_X$ trivially induces a functorial $P$-action $P\rightarrow End_{\mathcal{V}(X)}$.
In our first step we build a diagram

\[
\xymatrix{ & End_{\mathcal{Y}(P)}\ar[r]^{ev_X}\ar@{->>}[d]_{{\pi}}^{\sim} & End_{\mathcal{Y}(X)}\ar[d]\\
P\ar@{-->}[ur]\ar[r]^{=} & P\ar[r]^{ev_X} & End_{\mathcal{V}(X)}
}
\]
In $Ch_{\mathbb{K}}$, the endomorphism prop $End_{\mathcal{Y}(X)}$
is built via the two following pullbacks:

\medskip{}

\[
\xymatrix{End_{\mathcal{Y}(X)}\ar[r]\ar[d] & End_{\mathcal{Z}(X)}\ar[d]^{s^*\circ pr}\\
End_{X}\ar[r]_-{s_*} & Hom_{X,Z(X)}
}
\]
and

\medskip{}

\[
\xymatrix{End_{\mathcal{Z}(X)}\ar[r]\ar[d] & End_{X_0}\times End_{X_1}\ar[d]^{d_{0}^{*}\times d_1^{*}}\\
End_{Z(X)}\ar[r]_-{(d_0,d_1)_{*}} & Hom_{Z(X),X_0}\times Hom_{Z(X),X_1}
}
\]
where $s_*$ and $(d_0,d_1)_{*}$ are maps induced by the composition
by $s$ and $(d_0,d_1)$, and $s^*$, $d_0^*$, $d_1^*$
are maps induced by the precomposition by $s$, $d_0$ and $d_1$
. The projection $pr:End_{\mathcal{Z}(X)}\rightarrow End_{Z(X)}$
is induced by the inclusion of diagrams $\{Z(X)\}\hookrightarrow\{X_0\stackrel{\sim}{\twoheadleftarrow}Z(X)\stackrel{\sim}{\twoheadrightarrow}X_1\}$
(see \cite{Fre1}, Section 8). The idea is to define a prop of $P$-diagrams $End_{\mathcal{Y}(P)}$ with a form
similar to that of $End_{\mathcal{Y}(X)}$, in order to get the
prop morphism $ev_X:End_{\mathcal{Y}(P)}\rightarrow End_{\mathcal{Y}(X)}$
induced by the morphism $P\rightarrow End_X$ for each $X\in Ch_{\mathbb{K}}^P$.
For this aim we use two pullbacks
similar to those above with props of $P$-diagrams and
$\Sigma$-biobjects replacing the usual ones.

\textbf{Step 2.}
In our second step, we show that ${\pi}$ is an acyclic fibration in
$\mathcal{P}$ in order to obtain the desired lifting $P\rightarrow End_{\mathcal{Y}(P)}$
inducing natural $P$-actions
\[
P\rightarrow End_{\mathcal{Y}(P)} \rightarrow End_{\mathcal{Y}(X)}
\]
for every $X\in Ch_{\mathbb{K}}^P$, which respect the $P$-algebra structures on the diagrams $\mathcal{V}(X)$.
It endows the category of $P$-algebras with a functorial path object.
Finally, we prove Theorem 2.2 in Section 2.6, by using lifting properties in the category of props
and providing the desired zigzag of natural transformations $\varphi^*\stackrel{\sim}{\leftarrow}Z\stackrel{\sim}{\rightarrow}\psi^*$.
Then we show how to deduce Theorem 0.1.

\medskip{}

\begin{rem}
We can also wonder about the homotopy invariance of the classifying space up to Quillen equivalences.
Let $P$ be a prop in $\mathcal{E}_1$. Let $F:\mathcal{E}_1\rightleftarrows\mathcal{E}_2:G$
be a symmetric monoidal adjunction. The prop $F(P)$ is defined by
applying the functor $F$ entrywise to $P$: the fact that $F$ is
symmetric monoidal ensures the preservation of the composition products
of $P$, giving to $F(P)$ a prop structure. Lemma 7.1 of \cite{JY} says
that the adjoint pair $(F,G)$ induces an adjunction $\overline{F}:\mathcal{E}_1^P\rightleftarrows\mathcal{E}_2^{F(P)}:\overline{G}$.
Now suppose that $(F,G)$ forms a Quillen adjunction. By Brown's lemma,
the functor $F$ preserves weak equivalences between cofibrant objects
and the functor $G$ preserves weak equivalences between fibrant objects.
If all the objects of $\mathcal{E}_1$ and $\mathcal{E}_2$ are
fibrant and cofibrant, then the adjoint pair $(\overline{F},\overline{G})$
restricts to an adjunction $\overline{F}: w(\mathcal{E}_1)^P\rightleftarrows w(\mathcal{E}_2)^{F(P)}:\overline{G}$
and thus gives rise to a homotopy equivalence $\mathcal{N} w(\mathcal{E}_1)^{P}\sim\mathcal{N} w(\mathcal{E}_2)^{F(P)}$.
\end{rem}

\medskip{}

\subsection{The path object $Z(X)=Z\otimes X$}

Recall that in the model category structure of $Ch_{\mathbb{K}}$, the fibrations
are the degreewise surjections, the cofibrations are the degreewise injections,
and the weak equivalences are the morphisms inducing isomorphisms in homology.
The category $Ch_{\mathbb{K}}$ has moreover the simplifying feature that finite
products and coproducts coincide. Let $Z$ be the chain complex
defined by
\[
Z=\mathbb{K}\rho_0\oplus \mathbb{K}\rho_1\oplus \mathbb{K}\sigma_0\oplus \mathbb{K}\sigma_1\oplus \mathbb{K}\tau.
\]
The elements $\tau$, $\rho_0$ and $\rho_1$ are three generators of degree
$0$ and $\sigma_0$, $\sigma_1$ two generators of degree $-1$.
The differential $d_Z$ is defined by $d_Z(\sigma_0)=d_Z(\sigma_1)=0$,
$d_Z(\tau)=0$, $d_Z(\rho_0)=\sigma_0$ and $d_Z(\rho_1)=\sigma_1$.

\medskip{}

\begin{lem}
The chain complex $Z\otimes X$ defines a path object on $X$ in $Ch_{\mathbb{K}}$,
fitting in a factorization $X\stackrel{\sim}{\rightarrowtail_s}Z\otimes X\twoheadrightarrow_{(d_0,d_1)}X\oplus X$
of the diagonal $\Delta=(id_X,id_X):X\rightarrow X\oplus X$ such
that $s$ is an acyclic cofibration and $(d_0,d_1)$ a fibration.
\end{lem}
\medskip{}

\begin{proof}
Let $s:X\rightarrow Z\otimes X$ be the map defined by $s(x)=\tau\otimes x$.
Given the differential of $Z$, the map $s$ is clearly an injective
morphism of $Ch_{\mathbb{K}}$, i.e. a cofibration. We can also write $Z\otimes X\cong(\tilde{Z}\otimes X)\oplus X$
where
\[
\tilde{Z}=\mathbb{K}\rho_0\oplus \mathbb{K}\rho_1\oplus \mathbb{K}\sigma_0\oplus \mathbb{K}\sigma_1
\]
 is an acyclic complex. The acyclicity of $\tilde{Z}$ implies that
$s$ is an acyclic cofibration. We now define a map $(d_0,d_1):Z\otimes X\twoheadrightarrow X\oplus X$
such that $(d_0,d_1)\circ s=(id_X,id_X)$ and $(d_0,d_1)$
is a fibration. The map $d_0$ is determined for every $x\in X$
by $d_0(\tau\otimes x)=x$ and $d_0(\sigma_0\otimes x)=d_0(\sigma_1\otimes x)=d_0(\rho_0\otimes x)=d_0(\rho_1\otimes x)=0$.
The map $d_1$ is determined for every $x\in X$ by $d_1(\rho_0\otimes x)=x$,
$d_1(\tau\otimes x)=x$, $d_1(\sigma_0\otimes x)=d_1(\sigma_1\otimes x)=d_1(\rho_1\otimes x)=0$.
The map $(d_0,d_1)$ is clearly a surjective chain complexes morphism,
i.e. a fibration, and satisfies the equality $(d_0,d_1)\circ s=(id_X,id_X)$.
\end{proof}
\medskip{}
The two advantages of this path object on $X$ are its writing in
the form of a tensor product with $X$ and its decomposition into a
direct sum of $X$ with an acyclic complex.

\medskip{}

\subsection{The prop ${End}_{Z(P)}$}

Consider the endomorphism prop of $Z(X)$:

\begin{align*}
End_{Z(X)}(m,n) & =Hom_{Ch_{\mathbb{K}}}(Z(X)^{\otimes m},Z(X)^{\otimes n})\\
 & \cong Hom_{Ch_{\mathbb{K}}}(Z^{\otimes m}\otimes X^{\otimes m},Z^{\otimes n}\otimes X^{\otimes n})\\
 & \cong(Z^{\otimes m})^*\otimes Z^{\otimes n}\otimes End_{X}(m,n).
\end{align*}
We define a prop of $P$-diagrams such that
\begin{align*}
{End}_{Z(P)}(m,n) & =(Z^{\otimes m})^*\otimes Z^{\otimes n}\otimes P(m,n)\\
 & =\bigoplus t_1^*\otimes...\otimes t_m^*\otimes t_1\otimes...\otimes t_n\otimes P(m,n),
\end{align*}
where $t_i\in\{\rho_0,\rho_1,\sigma_0,\sigma_1,\tau\}$, together
with the following structure maps:

\medskip{}
\textbf{-Vertical composition product.} Let
\[
\alpha\in t_1^*\otimes...\otimes t_k^*\otimes t_1\otimes...\otimes t_n\otimes P(k,n)
\]
and
\[
\beta\in u_1^*\otimes...\otimes u_m^*\otimes u_1\otimes...\otimes u_k\otimes P(m,k).
\]
We set
\[
\alpha\circ_v\beta=\begin{cases}
\alpha\circ_v^P\beta & \text{if $(u_1,...,u_k)=(t_1,...,t_k)$},\\
0 & \text{otherwise},
\end{cases}
\]
where $\circ_v^P$ is the vertical composition product of $P$.

\medskip{}
\textbf{-Horizontal product.} Let
\[
\alpha\in t_1^*\otimes...\otimes t_{m_1}^*\otimes t_1\otimes...\otimes t_{n_1}\otimes P(m_1,n_1)
\]
and
\[
\beta\in u_1^*\otimes...\otimes u_{m_2}^*\otimes u_1\otimes...\otimes u_{n_2}\otimes P(m_2,n_2).
\]
We set
\begin{eqnarray*}
\alpha\circ_h\beta & = & t_1^*\otimes...\otimes t_{m_1}^*\otimes u_1^*\otimes...\otimes u_{m_2}^*\\
 &  & \otimes t_1\otimes...\otimes t_{n_1}\otimes u_1\otimes...\otimes u_{n_2}\otimes(\alpha\mid_{P(m_1,n_1)}\circ_h^P\beta\mid_{P(m_2,n_2)})\\
 & \in & t_1^*\otimes...\otimes t_{m_1}^*\otimes u_1^*\otimes...\otimes u_{m_2}^*\\
 &  & \otimes t_1\otimes...\otimes t_{n_1}\otimes u_1\otimes...\otimes u_{n_2}\otimes P(m_1+n_1,m_2+n_2),
\end{eqnarray*}
where $\circ_h^P$ is the horizontal product of $P$.

\medskip{}
\textbf{-Actions of the symmetric groups.} Let $\alpha=t_1^*\otimes...\otimes t_m^*\otimes t_1\otimes...\otimes t_n\otimes\alpha_P\in{End}_{Z(P)}(m,n)$
with $\alpha_P\in P(m,n)$. The action of a permutation $\sigma\in\Sigma_m$
on the right of this prop element is given by $\alpha.\sigma=t_{\sigma(1)}^*\otimes...\otimes t_{\sigma(m)}^*\otimes t_1\otimes...\otimes t_n\otimes\alpha_P.\sigma$.
The action of a permutation $\tau\in\Sigma_n$ on the left of this
prop element is given by $\tau.\alpha=t_1^*\otimes...\otimes t_m^*\otimes t_{\tau^{-1}(1)}\otimes...\otimes t_{\tau^{-1}(n)}\otimes\tau.\alpha_P$.

\medskip{}

Let $X\in Ch_{\mathbb{K}}^P$ be a $P$-algebra. From the definition of ${End}_{Z(P)}(m,n)$,
we easily see that the prop morphism $P\rightarrow End_X$ induces a prop
morphism
\[
ev_X:End_{Z(P)}\rightarrow End_{Z(X)}
\]
satisfying the appropriate functoriality diagrams (see Section 2.1.3).

\medskip{}

\subsection{The prop $End_{\mathcal{Z}(P)}$}

\subsubsection{The pullback defining $End_{\mathcal{Z}(X)}$ and its explicit maps}

For every $(m,n)\in\mathbb{N}^2$, we have a pullback

\medskip{}

\[
\xymatrix{End_{\mathcal{Z}(X)}(m,n)\ar[r]\ar[d] & End_{X_0}(m,n)\oplus End_{X_1}(m,n)\ar[d]^{(d_0^{\otimes m})^*\oplus(d_1^{\otimes m})^*}\\
End_{Z(X)}(m,n)\ar[r]_-{(d_0^{\otimes n},d_1^{\otimes n})_*} & Hom_{Z(X),X_0}(m,n)\oplus Hom_{Z(X),X_1}(m,n)
}.
\]
For every $X\in Ch_{\mathbb{K}}^P$ and $(m,n)\in\mathbb{N}^2$ we have the isomorphisms
\begin{align*}
Hom_{X,Z(X)}(m,n) & =Hom_{Ch_{\mathbb{K}}}(X^{\otimes m},Z(X)^{\otimes n})\\
 & \cong Hom_{Ch_{\mathbb{K}}}(X^{\otimes m},Z^{\otimes n}\otimes X^{\otimes n})\\
 & \cong Z^{\otimes n}\otimes End_{X}(m,n)
\end{align*}
and
\begin{align*}
Hom_{Z(X),X_i}(m,n) & =Hom_{Ch_{\mathbb{K}}}(Z(X)^{\otimes m},X^{\otimes n})\\
 & \cong Hom_{Ch_{\mathbb{K}}}(Z^{\otimes m}\otimes X^{\otimes m},X^{\otimes n})\\
 & \cong(Z^{\otimes m})^*\otimes End_{X_i}(m,n).
\end{align*}
Applying these isomorphisms, we get a pullback

\medskip{}

\[
\xymatrix{End_{\mathcal{Z}(X)}(m,n)\ar[r]\ar[d] & End_{X_0}(m,n)\oplus End_{X_1}(m,n)\ar[d]^{\overline{(d_0^{\otimes m})^*}\oplus\overline{(d_1^{\otimes m})^*}}\\
(Z^{\otimes m})^*\otimes Z^{\otimes n}\otimes End_X(m,n)\ar[r]_-{\overline{(d_0^{\otimes n},d_1^{\otimes n})_*}} & (Z^{\otimes m})^*\otimes End_{X_0}(m,n)\oplus(Z^{\otimes m})^*\otimes End_{X_1}(m,n)
}.
\]

\medskip{}
We have to make explicit the maps $\overline{(d_0^{\otimes n},d_1^{\otimes n})_*}$
and $\overline{(d_0^{\otimes m})^*}\oplus\overline{(d_1^{\otimes m})^*}$
and replace $End_{X_0}(m,n)$, $End_{X_1}(m,n)$ and $End_X(m,n)$
by $P_0(m,n)$, $P_1(m,n)$ and $P(m,n)$ to obtain a prop of $P$-diagrams $\{End_{\mathcal{Z}(P)}(m,n)\}_{(m,n)\in\mathbb{N}^2}$
acting naturally on $End_{\mathcal{Z}(X)}(m,n)$,
$X\in Ch_{\mathbb{K}}^P$. Then we apply the same method to build a prop of $P$-diagrams $End_{\mathcal{Y}(P)}$ acting naturally
on $End_{\mathcal{Y}(X)}$, $X\in Ch_{\mathbb{K}}^P$.

\medskip{}

\begin{lem}
Let $\{\underline{z}_i\}_{i\in I}$ be a basis of $Z^{\otimes m}$. The map
\[
\overline{(d_1^{\otimes m})^*}:End_X(m,n)\rightarrow(Z^{\otimes m})^*\otimes End_X(m,n)
\]
 is defined by the formula
\[
\overline{(d_1^{\otimes m})^*}(\xi)=\sum_{j\in J}(\underline{z}_j^*\otimes\xi)=(\sum_{j\in J}\underline{z}_j^*)\otimes\xi,
\]
where $J$ is the subset of $I$ such that $d_1^{\otimes m}(\underline{z}_j\otimes \underline{x})=\underline{x}$
for $\underline{x}\in X^{\otimes m}$ and $j\in J$.
\end{lem}

\medskip{}

\begin{proof}
First we give an explicit inverse to the well known isomorphism
\begin{align*}
\lambda:U^*\otimes Hom_{Ch_{\mathbb{K}}}(V,V')\stackrel{\cong}{\rightarrow} & Hom_{Ch_{\mathbb{K}}}(U\otimes V,V')\\
\varphi\otimes f\mapsto & [u\otimes v\mapsto\varphi(u).f(v)]
\end{align*}
where $U$ is supposed to be of finite dimension. Let $\{u_i\}_{i\in I}$
be a basis of $U$. We have $\lambda=\sum_{i\in I}\lambda_i$
where
\begin{align*}
\lambda_i:\mathbb{K}u_i^*\otimes Hom_{Ch_{\mathbb{K}}}(V,V')\rightarrow & Hom_{Ch_{\mathbb{K}}}(\mathbb{K}u_i\otimes V,V')\\
u_i^*\otimes f\mapsto & u_i^*.f:u_i\otimes v\mapsto u_i^*(u_i).f(v)=f(v)
\end{align*}
so
\begin{align*}
\lambda^{-1}:Hom_{Ch_{\mathbb{K}}}(U\otimes V,V')\rightarrow & U^*\otimes Hom_{Ch_{\mathbb{K}}}(V,V')\\
f\mapsto & \sum_{i\in I}(u_i^*\otimes f\mid_{\mathbb{K}u_i\otimes V}).
\end{align*}

Let $\sigma:Z^{\otimes m}\otimes X^{\otimes m}\rightarrow(Z\otimes X)^{\otimes m}$
be the map permuting the variables. Recall that the map $d_1$ is determined for every $x\in X$ by $d_1(\rho_0\otimes x)=x$,
$d_1(\tau\otimes x)=x$, $d_1(\sigma_0\otimes x)=d_1(\sigma_1\otimes x)=d_1(\rho_1\otimes x)=0$.
The map
\[
\overline{(d_1^{\otimes m})^*}:Hom_{Ch_{\mathbb{K}}}(X^{\otimes m},X^{\otimes n})\rightarrow Hom_{Ch_{\mathbb{K}}}(Z^{\otimes m}\otimes X^{\otimes m},X^{\otimes n})\stackrel{\cong}{\rightarrow}(Z^{\otimes m})^*\otimes Hom_{Ch_{\mathbb{K}}}(X^{\otimes m},X^{\otimes n})
\]
 is defined by
\[
\xi\mapsto\xi\circ d_1^{\otimes m}\circ\sigma\mapsto\sum_{i\in I}(\underline{z}_i^*\otimes(\xi\circ d_1^{\otimes m}\circ\sigma)\mid_{\mathbb{K}\underline{z}_i\otimes V}).
\]
We obtain finally
\begin{align*}
\overline{(d_1^{\otimes m})^*}:End_X(m,n)\rightarrow & (Z^{\otimes m})^*\otimes End_{X}(m,n)\\
\xi\mapsto & \sum_{j\in J}(\underline{z}_j^*\otimes\xi)=(\sum_{j\in J}\underline{z}_j^*)\otimes\xi
\end{align*}
where $J$ is the subset of $I$ such that $d_1^{\otimes m}(\underline{z}_j\otimes \underline{x})=\underline{x}$
for $\underline{x}\in X^{\otimes m}$ and $j\in J$. If $j\notin J$ then $d_1^{\otimes m}\mid_{\mathbb{K}\underline{z}_j\otimes X^{\otimes m}}=0$.
\end{proof}

\medskip{}

Recall that the map $d_0:Z\otimes X\rightarrow X$ is defined for every $x\in X$
by $d_0(\tau\otimes x)=x$ and $d_0(\sigma_0\otimes x)=d_0(\sigma_1\otimes x)=d_0(\rho_0\otimes x)=d_0(\rho_1\otimes x)=0$.
As previously, the map $\overline{(d_0^{\otimes m})^*}$ has a form similar
to that of $\overline{(d_1^{\otimes m})^*}$, and we have determined
$\overline{(d_0^{\otimes m})^*}\oplus\overline{(d_1^{\otimes m})^*}$.

\medskip{}

\begin{lem}
The map $\overline{(d_0^{\otimes n},d_1^{\otimes n})_*}$ is determined by
\[
\overline{(d_0^{\otimes n},d_1^{\otimes n})_*}:\underline{z}_j^*\otimes \underline{z}_i'\otimes\xi\mapsto\sum_{k\in I}(\underline{z}_k^*\otimes((d_0^{\otimes n},d_1^{\otimes n})\circ \underline{z}_j^*(-).\underline{z}_i'\otimes\xi)\mid_{\mathbb{K}\underline{z}_k\otimes X^{\otimes m}}).
\]
\end{lem}

\begin{proof}
Let $\{\underline{z}'_i\}_{i\in I'}$ be the basis of $Z^{\otimes n}$. We have the isomorphism
\begin{align*}
(Z^{\otimes m})^*\otimes Z^{\otimes n}\otimes Hom_{Ch_{\mathbb{K}}}(X^{\otimes m},X^{\otimes n})\rightarrow & Hom_{Ch_{\mathbb{K}}}(Z^{\otimes m}\otimes X^{\otimes m},Z^{\otimes n}\otimes X^{\otimes n})\\
\underline{z}_j^*\otimes \underline{z}_i'\otimes\xi\mapsto & \underline{z}_j^*(-).\underline{z}_i'\otimes\xi
\end{align*}
that we compose with
\begin{align*}
(d_0^{\otimes n},d_1^{\otimes n}):Z^{\otimes n}\otimes X^{\otimes n}\rightarrow & X_0^{\otimes n}\oplus X_1^{\otimes n}\\
\underline{z}_j\otimes x\mapsto & \begin{cases}
x\oplus x & \text{if $j\in J'$},\\
x\oplus0\: or\:0\oplus x & \text{otherwise},
\end{cases}
\end{align*}
where $J'$ is the subset of $I$ such that $d_0\mid_{\mathbb{K}\underline{z}_j\otimes X^{\otimes n}}\neq0$
and $d_1\mid_{\mathbb{K}\underline{z}_j\otimes X^{\otimes n}}\neq0$ for $j\in J'$.
Finally we compose with the isomorphism
\begin{align*}
Hom_{Ch_{\mathbb{K}}}(Z^{\otimes m}\otimes X^{\otimes m},X_0^{\otimes n}\oplus X_1^{\otimes n})\stackrel{\cong}{\rightarrow} & (Z^{\otimes m})^*\otimes Hom_{Ch_{\mathbb{K}}}(X^{\otimes m},X_0^{\otimes n}\oplus X_1^{\otimes n})\\
f\mapsto & \sum_{i\in I}(\underline{z}_i^*\otimes f\mid_{Kz_i\otimes X^{\otimes m}})
\end{align*}
and get the map
\[
\overline{(d_0^{\otimes n},d_1^{\otimes n})_*}:\underline{z}_j^*\otimes \underline{z}_i'\otimes\xi\mapsto\sum_{k\in I}(\underline{z}_k^*\otimes((d_0^{\otimes n},d_1^{\otimes n})\circ \underline{z}_j^*(-).\underline{z}_i'\otimes\xi)\mid_{\mathbb{K}\underline{z}_k\otimes X^{\otimes m}}).
\]
\end{proof}

\subsubsection{The associated prop of $P$-diagrams}

The key observation is that these two maps $\overline{(d_0^{\otimes m})^*}\oplus\overline{(d_1^{\otimes m})^*}$
and $\overline{(d_0^{\otimes n},d_1^{\otimes n})_*}$, fixing
the prop structure on $End_{\mathcal{Z}(X)}(m,n)$
in function of those of $(Z^{\otimes m})^*\otimes Z^{\otimes n}\otimes End_X(m,n)$
and $End_{X_0}(m,n)\oplus End_{X_1}(m,n)$, do not modify the
operations $\xi\in End_X(m,n)$ themselves. Therefore, we replace
$End_{X_0}(m,n)$, $End_{X_1}(m,n)$ and $End_X(m,n)$ by $P_0(m,n)$,
$P_1(m,n)$ and $P(m,n)$ to get this new pullback

\medskip{}

\[
\xymatrix{End_{\mathcal{Z}(P)}(m,n)\ar[r]\ar[d] & P_0(m,n)\oplus P_1(m,n)\ar[d]^{\overline{(d_0^{\otimes m})^*}\oplus\overline{(d_1^{\otimes m})^*}}\\
(Z^{\otimes m})^*\otimes Z^{\otimes n}\otimes P(m,n)\ar[r]_-{\overline{(d_0^{\otimes n},d_1^{\otimes n})_*}} & (Z^{\otimes m})^*\otimes P_0(m,n)\oplus(Z^{\otimes m})^*\otimes P_1(m,n)
}.
\]
The explicit formulae of the applications defining this pullback, given by Lemmas 2.5 and 2.6, show that these replacements
do not break the prop structure transfer. Thus we get the desired prop of $P$-diagrams $End_{\mathcal{Z}(P)}$
having the same shape as that of $End_{\mathcal{Z}(X)}$
and thus acting naturally on the associated diagram of $P$-algebras via the evaluation morphism
\[
ev_X:End_{\mathcal{Z}(P)}\rightarrow End_{\mathcal{Z}(X)}.
\]

\subsection{The prop $End_{\mathcal{Y}(P)}$ and the functorial path object in $P$-algebras}

Now let us define $End_{\mathcal{Y}(P)}$. For every
$(m,n)\in\mathbb{N}^2$, the pullback

\medskip{}

\[
\xymatrix{End_{\mathcal{Y}(X)}(m,n)\ar[r]\ar[d] & End_{\mathcal{Z}(X)}(m,n)\ar[d]^{(s^{\otimes m})^*\circ pr}\\
End_X(m,n)\ar[r]_-{(s^{\otimes n})_*} & Hom_{X,Z(X)}(m,n)
}
\]
induces via the isomorphims explained at the beginning of 3.3 and 3.4.1 a pullback
\[
\xymatrix{End_{\mathcal{Y}(X)}(m,n)\ar[r]\ar[d] & End_{\mathcal{Z}(X)}(m,n)\ar[d]^{\overline{(s^{\otimes m})^*\circ pr}}\\
End_X(m,n)\ar[r]_-{\overline{(s^{\otimes n})_*}} & Z^{\otimes n}\otimes End_X(m,n)
}.
\]

\medskip{}

In the same manner as before, given that $s:X\rightarrow Z\otimes X$ sends every
$x\in X$ to $\tau\otimes x$, the map $\overline{(s^{\otimes m})^*}$
is of the form
\begin{align*}
(Z^{\otimes m})^*\otimes Z^{\otimes n}\otimes End_{X}(m,n)\rightarrow & Z^{\otimes n}\otimes End_{X}(m,n)\\
\underline{z}_j^*\otimes \underline{z}_i'\otimes\xi\mapsto & \begin{cases}
\underline{z}_i'\otimes\xi & \text{if $j\in K$},\\
0 & \text{otherwise},
\end{cases}
\end{align*}
where $K$ is a certain subset of $I$ and $\overline{(s^{\otimes n})_*}$
is of the form
\begin{align*}
End_X(m,n)\rightarrow & Z^{\otimes n}\otimes End_X(m,n)\\
\xi\mapsto & \sum_{i\in K'}\underline{z}_i'\otimes\xi
\end{align*}
where $K'$ is a certain subset of $I'$. These two maps $\overline{(s^{\otimes m})^*\circ pr}$
and $\overline{(s^{\otimes n})_*}$, fixing the prop structure on
$End_{\mathcal{Y}(X)}(m,n)$ in function of those of $End_{\mathcal{Z}(X)}(m,n)$
and $End_X(m,n)$, do not modify the operations $\xi\in End_X(m,n)$
themselves. Therefore, we replace $End_X(m,n)$ by $P(m,n)$ and
$End_{\mathcal{Z}(X)}(m,n)$
by $End_{\mathcal{Z}(P)}(m,n)$
to get this new pullback

\medskip{}

\[
\xymatrix{End_{\mathcal{Y}(P)}(m,n)\ar[r]\ar[d] & End_{\mathcal{Z}(P)}(m,n)\ar[d]^{\overline{(s^{\otimes m})^*\circ pr}}\\
P(m,n)\ar[r]_-{\overline{(s^{\otimes n})_*}} & Z^{\otimes n}\otimes P(m,n)
}.
\]
The explicit formulae of the applications defining this pullback show that these replacements
do not break the prop structure transfer. Thus we get the desired prop of $P$-diagrams $End_{\mathcal{Y}(P)}$ having
the same shape as that of $End_{\mathcal{Y}(X)}$ and thus acting
naturally on the associated diagram of $P$-algebras via the evaluation morphism
\[
ev_X:End_{\mathcal{Y}(P)}\rightarrow End_{\mathcal{Y}(X)}.
\]
We finally obtain the following lemma:

\medskip{}

\begin{lem}
There is a commutative diagram of props
\[
\xymatrix{ & End_{\mathcal{Y}(P)}\ar[r]^{ev_X}\ar[d]_{\pi} & End_{\mathcal{Y}(X)}\ar[d]\\
P\ar[r]^{=} & P\ar[r] & End_{\mathcal{V}(X)}
}
\]

\end{lem}
\medskip{}

Now we want to prove that the morphism $P\rightarrow End_{\mathcal{V}(X)}$
lifts to a morphism $P\rightarrow End_{\mathcal{Y}(P)} \stackrel{ev_X}{\rightarrow} End_{\mathcal{Y}(X)}$:

\medskip{}

\begin{lem}
The map $\pi$ is an acyclic fibration in the category
of props.
\end{lem}

\medskip{}

\begin{proof}
According to the model category structure on $\mathcal{P}$, it is
sufficient to prove that for every $(m,n)\in\mathbb{N}^2$, $\pi(m,n)$
is an acyclic fibration of chain complexes. The map $\pi(m,n)$
is given by the base extension

\[
\pi(m,n)=P(m,n)\bigtimes_{Hom_{P,Z(P)}(m,n)}\phi(m,n)\bigtimes_{Hom_{Z(P),P_0}(m,n)\oplus Hom_{Z(P),P_1}(m,n)}(P_0(m,n)\oplus P_1(m,n))
\]
where
\[
\phi(m,n):End_{Z(P)}(m,n)\rightarrow Hom_{P,Z(P)}(m,n)\bigtimes_{P_0(m,n)\oplus P_1(m,n)}(Hom_{Z(P),P_0}(m,n)\oplus Hom_{Z(P),P_1}(m,n))
\]
comes from the diagram

\medskip{}

\[
\xymatrix{End_{Z(P)}(m,n)\ar[dr]^{\phi(m,n)}\ar@/_{1pc}/[ddr]_{\overline{(s^{\otimes m})^*}}\ar@/^{1pc}/[drr]^{\overline{(d_0^{\otimes n},d_1^{\otimes n})_*}}\\
 & pullback\ar[d]\ar[r] & Hom_{Z(P),P_0}(m,n)\oplus Hom_{Z(P),P_1}(m,n))\ar[d]^{\overline{(s^{\otimes m})^*}\oplus\overline{(s^{\otimes m})^*}}\\
 & Hom_{P,Z(P)}(m,n)\ar[r]_-{\overline{(d_0^{\otimes n},d_1^{\otimes n})_*}} & P_0(m,n)\oplus P_1(m,n)
}
\]
i.e.

\medskip{}

\[
\xymatrix{(Z^{\otimes m})^*\otimes Z^{\otimes n}\otimes P(m,n)\ar[dr]^{\phi(m,n)}\ar@/_{1pc}/[ddr]_{\overline{(s^{\otimes m})^*}}\ar@/^{1pc}/[drr]^{\overline{(d_0^{\otimes n},d_1^{\otimes n})_*}}\\
 & pullback\ar[d]\ar[r] & (Z^{\otimes m})^*\otimes(P_0(m,n)\oplus P_1(m,n))\ar[d]^{\overline{(s^{\otimes m})^*}\oplus\overline{(s^{\otimes m})^*}}\\
 & Z^{\otimes n}\otimes P(m,n)\ar[r]_-{\overline{(d_0^{\otimes n},d_1^{\otimes n})_*}} & P_0(m,n)\oplus P_1(m,n)
}.
\]

\medskip{}

We have an isomorphism
\begin{align*}
P_0(m,n)\oplus P_1(m,n)\stackrel{\cong}{\rightarrow} & (\mathbb{K}p_0\oplus \mathbb{K}p_1)\otimes P(m,n)\\
p\oplus p'\mapsto & p_0\otimes p+p_1\otimes p
\end{align*}
where $p_0$ and $p_1$ are two generators of degree $0$. The previous computations give
\begin{align*}
\overline{(d_0^{\otimes n},d_1^{\otimes n})_*}:Z^{\otimes n}\otimes P(m,n)\rightarrow & (\mathbb{K}p_0\oplus \mathbb{K}p_1)\otimes P(m,n)\\
\underline{z}_i'\otimes p\mapsto & \begin{cases}
(p_0\oplus p_1)\otimes p & \text{if $i\in J'$},\\
p_0\otimes p\: or\: p_1\otimes p & \text{otherwise},
\end{cases}
\end{align*}
and the map
\[
\overline{(s^{\otimes m})^*}\oplus\overline{(s^{\otimes m})^*}:(Z^{\otimes m})^*\otimes(\mathbb{K}p_0\oplus \mathbb{K}p_1)\otimes P(m,n)\rightarrow(\mathbb{K}p_0\oplus \mathbb{K}p_1)\otimes P(m,n)
\]
is defined by
\[
\underline{z}_j^*\otimes(\lambda p_0\oplus\mu p_1)\otimes p\mapsto\begin{cases}
(\lambda p_0\oplus\mu p_1)\otimes p\: or\:\lambda p_0\otimes p\: or\:\mu p_1\otimes p, & \text{if $j\in K$},\\
0=0\otimes p & \text{otherwise}.
\end{cases}
\]
We have similar results for the two maps starting from $(Z^{\otimes m})^*\otimes Z^{\otimes n}\otimes P(m,n)$.
We deduce that the previous diagram is the image under the functor
$-\otimes P(m,n)$ of the dual pushout-product

\medskip{}

\[
\xymatrix{Hom_{Ch_{\mathbb{K}}}(Z^{\otimes m},Z^{\otimes n})\ar[dr]^{(f_s^*,(g_{d_0,d_1})_*)}\ar@/_{1pc}/[ddr]_{f_s^*}\ar@/^{1pc}/[drr]^{(g_{d_0,d_1})_*}\\
 & pullback\ar[d]\ar[r] & Hom_{Ch_{\mathbb{K}}}(Z^{\otimes m},\mathbb{K}p_0\oplus \mathbb{K}p_1)\ar[d]^{f_s^*}\\
 & Hom_{Ch_{\mathbb{K}}}(\mathbb{K},Z^{\otimes n})\ar[r]_-{(g_{d_0,d_1})_*} & Hom_{Ch_{\mathbb{K}}}(\mathbb{K},\mathbb{K}p_0\oplus \mathbb{K}p_1)
}
\]
modulo the isomorphisms

\[
Z^{\otimes n}\cong Hom_{Ch_{\mathbb{K}}}(\mathbb{K},Z^{\otimes n}),
\]

\[
(Z^{\otimes m})^*\otimes Z^{\otimes n}\cong Hom_{Ch_{\mathbb{K}}}(Z^{\otimes m},Z^{\otimes n}),
\]

\[
(Z^{\otimes m})^*\otimes(\mathbb{K}p_0\oplus \mathbb{K}p_1)\cong Hom_{Ch_{\mathbb{K}}}(Z^{\otimes m},\mathbb{K}p_0\oplus \mathbb{K}p_1)
\]
and
\[
\mathbb{K}p_0\oplus \mathbb{K}p_1\cong Hom_{Ch_{\mathbb{K}}}(\mathbb{K},\mathbb{K}p_0\oplus \mathbb{K}p_1).
\]
The map $g_{d_0,d_1}:Z^{\otimes n}\rightarrow Kp_0\oplus Kp_1$
is surjective so it is a fibration of chain complexes. Recall that
we have a decomposition of $Z$ into $Z=\widetilde{Z}\oplus \mathbb{K}\tau$ where
$\widetilde{Z}$ is acyclic, which implies a decomposition of $Z^{\otimes m}$
of the form $Z^{\otimes m}\cong S_m\oplus \mathbb{K}(\tau^{\otimes n})$ where
$S_m$ is acyclic because it is a sum of tensor products containing
$\widetilde{Z}$. The map $f_s$ is an injection sending $\mathbb{K}$ on
$\mathbb{K}(\tau^{\otimes n})$ so it is a cofibration, and $S_m$ is acyclic
so $f_s$ is an acyclic cofibration. Applying the axiom MM1' in
$Ch_{\mathbb{K}}$ we conclude that $(f_s^*,(g_{d_0,d_1})_*)$ is
an acyclic fibration. Therefore $\phi(m,n)=(f_s^*,(g_{d_0,d_1})_*)\otimes id_{P(m,n)}$
is an acyclic fibration, and so is $\pi(m,n)$, given that
the class of acyclic fibrations is stable by base extension.
\end{proof}

\medskip{}

We have proved the following result:

\begin{prop}
There is a functorial $P$-action $P\rightarrow End_{\mathcal{Y}(X)}$, and consequently a functorial path object $Z:(Ch_{\mathbb{K}})^P\rightarrow (Ch_{\mathbb{K}})^P$ in the category of cofibrant-fibrant $P$-algebras $(Ch_{\mathbb{K}})^P$.
\end{prop}

\subsection{Proof of the final result}

Consider now the square of inclusions of diagrams
\[
\xymatrix{\mathcal{T}(X)\ar@{^{(}->}[d]_{u}\ar@{^{(}->}[r]^t & \mathcal{V}(X)\ar@{^{(}->}[d]^v\\
\mathcal{Z}(X)\ar@{^{(}->}[r]_w & \mathcal{Y}(X)
}
\]
where $\mathcal{V}(X)$, $\mathcal{Z}(X)$ and $\mathcal{Y}(X)$ are the diagrams defined previously and
$\mathcal{T}(X)$ is the diagram $\{X_0,X_1\}$ consisting of two copies of $X$ and no arrows between them.
This square of inclusions induces the following commutative square of endomorphism props

\medskip{}

\[
\xymatrix{End_{\mathcal{Y}(X)}\ar[r]^-{w_*}\ar[d]_{v_*} & End_{\mathcal{Z}(X)}\ar[d]^{u^*}\\
End_{\mathcal{V}(X)}\ar[r]_{t^*} & End_{\mathcal{T}(X)}
}
\]
where $u^*$, $v^*$, $t^*$ and $w^*$ are the maps induced
by the inclusions of the associated diagrams of $P$-algebras. We have
a commutative diagram of props of $P$-diagrams reflecting
this square

\medskip{}

\[
\xymatrix{End_{\mathcal{Y}(P)}\ar[r]^-{w_*}\ar[d]_{v_*} & End_{\mathcal{Z}(P)}\ar[d]^{u^*}\\
End_{\mathcal{V}(P)}=P\ar[r]_{t^*} & End_{\mathcal{T}(P)}=P_0\times P_1
}
\]
where $v^*$ is the acyclic fibration $\pi$ of Lemma
2.8 and $u^*$ is a fibration because it is clearly surjective in
each biarity (recall that the surjective morphisms are the fibrations
of $Ch_{\mathbb{K}}$ and that the fibrations of $\mathcal{P}$ are determined
componentwise). Now we can use this commutative square to prove the final result:

\medskip{}

\begin{thm}
Let $P$ be a cofibrant prop and $\varphi,\psi:P\rightarrow P$ two
homotopic prop morphisms, then there exists a diagram of functors
\[
\varphi^*\stackrel{\sim}{\leftarrow}Z\stackrel{\sim}{\rightarrow}\psi^*
\]
where $Z$ is the path object functor defined in Proposition 2.9 and the
natural transformations are pointwise acyclic fibrations.
\end{thm}
\medskip{}

\begin{proof}
This proof follows the arguments of the proof of Theorem 8.4 in \cite{Fre1}.
We consider a cylinder object of $P$ fitting in a diagram of the form:
\[
\xymatrix{P\vee P\ar@{>->}[]!R+<4pt,0pt>;[r]_{(d_0,d_1)} & \widetilde{P}\ar@{->>}[r]_{s_0}^{\sim} & P}
\]
The components $d_0$ and $d_1$ of the morphism $(d_0,d_1)$ are acyclic cofibrations
because $P$ is cofibrant by assumption (see Lemma 4.4 in \cite{DS}) and
$s_0$ an acyclic fibration. The fact that $\varphi$ and $\psi$
are homotopic implies the existence of a lifting in

\medskip{}

\[
\xymatrix{P\vee P\ar@{>->}[]!D-<0pt,4pt>;[d]_{(d_0,d_1)}\ar[r]^{(\varphi,\psi)} & P\ar[d]\\
\widetilde{P}\ar@{-->}[ur]^h\ar[r] & 0
}
\]

\medskip{}
We produce the lifting
\[
\xymatrix{I\ar@{>->}[]!D-<0pt,4pt>;[d]\ar[r] & End_{\mathcal{Y}(P)}\ar@{->>}[d]_{\sim}^{v^*}\\
P\ar@{-->}[ur]^k\ar[r]_{\varphi} & P
}
\]
(by the axiom MC4 of model categories, see \cite{DS}) and form
$(\varphi\circ s_0,h):\widetilde{P}\rightarrow P_0\times P_1$
in order to get the following commutative diagram:

\medskip{}

\[
\xymatrix{P\ar[r]^-k\ar@{>->}[]!D-<0pt,4pt>;[d]_{d_0}^{\sim} & End_{\mathcal{Y}(P)}\ar[r]^-{w^*} & End_{\mathcal{Z}(P)}\ar@{->>}[d]^{u^*}\\
\widetilde{P}\ar@{-->}[urr]^{l}\ar[rr]_{(\varphi\circ s_0,h)} &  & P_0\times P_1
}
\]
We have $(\varphi\circ s_0,h)\circ d_0=(\varphi\circ s_0\circ d_0,h\circ d_0)=(\varphi,\varphi)$
and $u^*\circ w^*\circ k=t^*\circ v^*\circ k=t^*\circ\varphi=(\varphi,\varphi)$
so this diagram is indeed commutative and there exists a lifting (axiom
MC4) $l:\widetilde{P}\rightarrow End_{\mathcal{Z}(P)}$.
Then we form $l\circ d_1:P\rightarrow End_{\mathcal{Z}(P)}$
and observe that $u^*\circ l\circ d_1=(\varphi\circ s_0,h)\circ d_1=(\varphi\circ s_0\circ d_1,h\circ d_1)=(\varphi,\psi)$,
i.e. we obtain the following diagram:

\medskip{}

\[
\xymatrix{ & End_{\mathcal{Z}(P)}\ar[r]^{ev_X}\ar[d]_{u^*} & End_{\mathcal{Z}(X)}\ar[d]\\
P\ar[r]_{(\varphi,\psi)}\ar[ur]^{l\circ d_1} & P_0\times P_1\ar[r]_{ev_X} & End_{\mathcal{T}(X)}
}
\]
and consequently a diagram of functors $\varphi^*\stackrel{\sim}{\twoheadleftarrow}Z\stackrel{\sim}{\twoheadrightarrow}\psi^*$.
The functorial path object $Z$ on $Ch_{\mathbb{K}}$ preserves weak equivalences
and restrict to an endofunctor of $wCh_{\mathbb{K}}$, so the associated functorial
path object $Z$ on $Ch_{\mathbb{K}}^P$ do the same. Moreover, the natural
transformations are weak equivalences in each component, so this diagram
restricts to the desired diagram of endofunctors of $wCh_{\mathbb{K}}^P$.
\end{proof}

\medskip{}

Now we can conclude the proof of Theorem 0.1 in the case $\mathcal{E}=Ch_{\mathbb{K}}$:
\medskip{}

\begin{thm}
Let $Ch_{\mathbb{K}}$ be the category of $\mathbb{Z}$-graded chain complexes
over a field $\mathbb{K}$ of characteristic zero. Let $\varphi:P\stackrel{\sim}{\rightarrow}Q$
be a weak equivalence between two cofibrant props. The map $\varphi$
gives rise to a functor $\varphi^*: w(Ch_{\mathbb{K}})^Q\rightarrow w(Ch_{\mathbb{K}})^P$
which induces a weak equivalence of simplicial sets $\mathcal{N}\varphi^*:\mathcal{N} w(Ch_{\mathbb{K}})^Q\stackrel{\sim}{\rightarrow}\mathcal{N} w(Ch_{\mathbb{K}})^P$.
\end{thm}
\medskip{}

\begin{proof}
Recall that $\mathcal{P}$ is the category of props in $Ch_{\mathbb{K}}$. Let us suppose first
that $\varphi:P\stackrel{\sim}{\rightarrowtail}Q$ is an acyclic cofibration
between two cofibrants props of $\mathcal{P}$. All objects in $Ch_{\mathbb{K}}$
are fibrant, so by definition of the model category structure on $\mathcal{P}$
the prop $P$ is fibrant and thus we have the following lifting

\medskip{}

\[
\xymatrix{P\ar[r]^{=}\ar@{>->}[]!D-<0pt,4pt>;[d]_{\varphi}^{\sim} & P\ar@{->>}[d]\\
Q\ar[r]\ar@{-->}[ur]^s & pt
}
\]
where $s:Q\stackrel{\sim}{\rightarrow}P$ satisfies
\[
\begin{cases}
s\circ\varphi=id_P\\
\varphi\circ s\sim id_Q
\end{cases}
\]
 (the relation $\sim$ is the homotopy relation for the model category
structure of $\mathcal{P}$). These maps induce functors $\varphi^*:(w\mathcal{E}^{cf})^Q\rightarrow(w\mathcal{E}^{cf})^P$
and $s^*:(w\mathcal{E}^{cf})^P\rightarrow(w\mathcal{E}^{cf})^Q$.
Applying the simplicial nerve functor, we obtain
\[
\begin{cases}
\mathcal{N}(s\circ\varphi)^*=\mathcal{N}\varphi^*\circ\mathcal{N}s^*=id_{(w\mathcal{E}^{cf})^P}\\
\mathcal{N}(\varphi\circ s)^*=\mathcal{N}s^*\circ\mathcal{N}\varphi^*\sim id_{(w\mathcal{E}^{cf})^Q}
\end{cases}
\]
so $\mathcal{N}\varphi^*$ is a homotopy equivalence in $sSet$,
which implies that it is a weak equivalence of simplicial sets . The
functor
\begin{align*}
\mathcal{P}\rightarrow & sSet\\
P\mapsto & \mathcal{N} w(\mathcal{E}^{cf})^P
\end{align*}
is defined between two model categories, and maps the acyclic cofibrations
between cofibrant objects to weak equivalences, so it preserves weak
equivalences between cofibrant objects according to Brown's lemma.
\end{proof}

\medskip{}

\subsection{The general case of a category $\mathcal{E}$ tensored over $Ch_{\mathbb{K}}$}

To complete our results we explain how the proof of Theorem 2.10 extends to a category
$\mathcal{E}$ tensored over $Ch_{\mathbb{K}}$.

\medskip{}

\begin{thm}
Let $\mathcal{E}$ be a cofibrantly generated symmetric monoidal model
category over $Ch_{\mathbb{K}}$. Let $\varphi:P\stackrel{{\sim}}{\rightarrow}Q$
be a weak equivalence between two cofibrant props defined in $Ch_{\mathbb{K}}$.
This morphism $\varphi$ gives rise to a functor $\varphi^*: w(\mathcal{E}^{c})^Q\rightarrow w(\mathcal{E}^{c})^P$
which induces a weak equivalence of simplicial sets $\mathcal{N}\varphi^*:\mathcal{N} w(\mathcal{E}^{c})^Q\rightarrow\mathcal{N} w(\mathcal{E}^{c})^P$.
\end{thm}

\medskip{}

\begin{proof}
The chain complex $Z$ defined previously is itself
the path object on $C^0$, so we have the commutative diagram

\medskip{}

\[
\xymatrix{ &  & C^{0}\\
C^0\ar@/^{1pc}/[urr]^{=}\ar@/^{-1pc}/[drr]_{=}\ar@{>->}[]!R+<4pt,0pt>;[r]_-{s}^-{\sim} & Z\ar@{->>}[ur]_{d_0}^{\sim}\ar@{->>}[dr]^{d_1}_{\sim}\\
 &  & C^0
}.
\]
Given that $C^0$ is the unit of $Ch_{\mathbb{K}}$, for any $X\in\mathcal{E}$
we have $C^0\otimes X\cong X$, thus by applying the functor
$-\otimes X$ we get the commutative diagram

\medskip{}

\[
\xymatrix{ &  & X_0\\
X\ar@/^{1pc}/[urr]^{=}\ar@/^{-1pc}/[drr]_{=}\ar@{>->}[]!R+<4pt,0pt>;[r]_-{s\otimes id_X}^-{\sim} & Z\otimes X\ar[ur]_{d_0\otimes id_X}^{\sim}\ar[dr]^{d_1\otimes id_X}_{\sim}\\
 &  & X_1
}.
\]

\medskip{}
The axiom MM1 for the external tensor product $\otimes$ implies
that if $X$ is cofibrant, then the functor $-\otimes X$ preserves
acyclic cofibrations of $Ch_{\mathbb{K}}$ (all the objects of $Ch_{\mathbb{K}}$ are
cofibrant) and thus, by Brown's lemma, it preserves the weak equivalences.
Therefore $s\otimes id_X$ is still an acyclic cofibration and
$d_0\otimes id_X$, $d_1\otimes id_X$ are weak equivalences.
Moreover, given the properties of $\otimes$ and the fact that
endomorphism props in $Ch_{\mathbb{K}}$ for objects of $\mathcal{E}$ are
defined with the external hom bifunctor $Hom_{\mathcal{E}}(-,-)$
of $\mathcal{E}$, we have the following isomorphisms:

\begin{align*}
End_{Z\otimes X}(m,n) & =Hom_{\mathcal{E}}((Z\otimes X)^{\otimes m},(Z\otimes X)^{\otimes n})\\
 & \cong Hom_{\mathcal{E}}(Z^{\otimes m}\otimes X^{\otimes m},Z^{\otimes n}\otimes X^{\otimes n})\\
 & \cong(Z^{\otimes m})^*\otimes Z^{\otimes n}\otimes End_X(m,n)
\end{align*}

\begin{align*}
Hom_{X,Z\otimes X}(m,n) & =Hom_{\mathcal{E}}(X^{\otimes m},(Z\otimes X)^{\otimes n})\\
 & \cong Hom_{\mathcal{E}}(X^{\otimes m},Z^{\otimes n}\otimes X^{\otimes n})\\
 & \cong Z^{\otimes n}\otimes End_X(m,n),
\end{align*}
and
\begin{align*}
Hom_{Z\otimes X,X_i}(m,n) & =Hom_{\mathcal{E}}((Z\otimes X)^{\otimes m},X^{\otimes n})\\
 & \cong Hom_{\mathcal{E}}(Z^{\otimes m}\otimes X^{\otimes m},X^{\otimes n})\\
 & \cong(Z^{\otimes m})^*\otimes End_{X_i}(m,n).
\end{align*}
The proofs of 2.3, 2.4 and 2.5 extend without changes to the case of a category $\mathcal{E}$ tensored over $Ch_{\mathbb{K}}$:
we still work in $Ch_{\mathbb{K}}$, and as before the operations associated to $s\otimes id_X$,
$d_0\otimes id_X$ and $d_1\otimes id_X$ in the pullbacks
do not transform the elements of $End_X(m,n)$ themselves, so that
the replacement of $End_X(m,n)$ by $P(m,n)$ does not break the
transfer of prop structure in these pullbacks. We obtain a diagram
of endofunctors $\varphi^*\stackrel{\sim}{\leftarrow}Z\stackrel{\sim}{\rightarrow}\psi^*$
of $(\mathcal{E}^{c})^P$ where the natural transformations are
weak equivalences in each component, so this diagram restricts to the
desired diagram of endofunctors of $w(\mathcal{E}^{c})^P$.
Theorem 0.1 is proved in the general case.
\end{proof}

\bigskip{}

\section{Extension of the results in the colored prop setting}

\bigskip{}

\begin{defn}
Let $C$ be a non-empty set, called the \emph{set of colors}, and $\mathcal{C}$ be
a symmetric monoidal category.

(1) A \emph{$C$-colored $\Sigma$-biobject} $M$ is a double sequence of
objects $\{M(m,n)\in\mathcal{E}\}_{(m,n)\in\mathbb{N}^2}$ where
each $M(m,n)$ admits commuting left $\Sigma_m$-action and right
$\Sigma_n$-action as well as a decomposition
\[
M(m,n)=colim_{c_i,d_i\in C}M(c_1,...,c_m;d_1,...,d_n)
\]
compatible with these actions. The objects $M(c_1,...,c_m;d_1,...,d_n)$
should be thought as spaces of operations with colors $c_1,...,c_m$
indexing the $m$ inputs and colors $d_1,...,d_n$ indexing the
$n$ outputs.

(2) A \emph{$C$-colored prop $P$} is a $C$-colored $\Sigma$-biobject
endowed with a horizontal composition
\begin{align*}
\circ_h:P(c_{11},...,c_{1m_1};d_{11},...,d_{1n_1})\otimes...\otimes P(c_{k1},...,c_{km_k};d_{k1},...,d_{kn_1}) & \rightarrow\\
P(c_{11},...,c_{km_k};d_{k1},...,d_{kn_k})\subseteq P(m_1+...+m_k,n_1+...+n_k)\\
\end{align*}
and a vertical composition
\[
\circ_v:P(c_1,...,c_k;d_1,...,d_n)\otimes P(a_1,...,a_m;b_1,...,b_k)\rightarrow P(a_1,...,a_m;d_1,...,d_n)\subseteq P(m,n)
\]
which is equal to zero unless $b_i=c_i$ for $1\leq i\leq k$.
These two compositions satisfy associativity axioms (we refer the
reader to \cite{JY} for details).
\end{defn}

\medskip{}

\begin{defn}
(1) Let $\{X_c\}_C$ be a collection of objects of $\mathcal{E}$.
The \emph{$C$-colored endomorphism prop} $End_{\{X_c\}_{C}}$ is defined
by
\[
End_{\{X_c\}_{C}}(c_1,...,c_m;d_1,...,d_n)=Hom_{\mathcal{E}}(X_{c_1}\otimes...\otimes X_{c_m},X_{d_1}\otimes...\otimes X_{d_n})
\]
with horizontal composition given by the tensor product of homomorphisms
and vertical composition given by the composition of homomorphisms
with matching colors.

(2) Let $P$ be a $C$-colored prop. A \emph{$P$-algebra} is the data of
a collection of objects $\{X_c\}_{C}$ and a $C$-colored prop morphism
$P\rightarrow End_{\{X_c\}_{C}}$.
\end{defn}

\medskip{}

\begin{example}
Let $I$ be a small category, $P$ a prop in $\mathcal{C}$. We can
build an $ob(I)$-colored prop $P_I$ such that the $P_I$-algebras
are the $I$-diagrams of $P$-algebras in $\mathcal{E}$ in the same
way as that of \cite{Mar1}.
\end{example}

\medskip{}

To endow the category of colored props with a model category structure,
the cofibrantly generated symmetric monoidal model structure on $\mathcal{C}$
is not sufficient. We have to suppose moreover that the domains of
the generating cofibrations and acyclic generating cofibrations are
small (cf \cite{Hir1}, 10.4.1), that is to say, the model structure
is strongly cofibrantly generated:

\medskip{}

\begin{thm}
(cf. \cite{JY}, Theorem 1.1) Let $C$ be a non-empty set. Let
$\mathcal{C}$ be a strongly cofibrantly generated symmetric monoidal
model category with a symmetric monoidal fibrant replacement functor,
and either:

(1) a cofibrant unit and a cocommutative interval, or

(2) functorial path data.

Then the category $\mathcal{P}_C$ of $C$-colored props in $\mathcal{C}$
forms a strongly cofibrantly generated model category with fibrations
and weak equivalences defined componentwise in $\mathcal{C}$.
\end{thm}

\medskip{}

This theorem works in particular with the categories of simplicial sets,
simplicial modules over a commutative ring, and chain complexes over
a ring of characteristic $0$ (our main category in this paper).

\medskip{}

This model structure is similar to that of $1$-colored props, and
we can define $C$-colored endomorphism props of morphisms (see \cite{JY},
Section 4) and more generally of any kind of diagram, so the lifting
properties used in the previous section work in the $C$-colored
case. Moreover, in the proof of Theorem 0.1, the replacement of the
operations $X^{\otimes m}\rightarrow X^{\otimes n}$ by $P(m,n)$
can be done using a $C$-colored prop $P$ instead of a $1$-colored
one without changing anything to the proof, therefore we finally get
the $C$-colored version of Theorem 0.1. We do
not have to change anything to Theorem 0.1, given that $Ch_{\mathbb{K}}$ satisfies
the hypotheses of Theorem 3.4.
The colored version of Corollary 0.3 then follows in the same way.

\section{Afterword}

Let $\mathcal{E}$ be a symmetric monoidal model category over $Ch_{\mathbb{K}}$.
Let $P$ be a cofibrant prop defined in $Ch_{\mathbb{K}}$ and $X$ an object of $\mathcal{E}$.
One can consider the moduli space $P\{X\}$ of $P$-algebra structures on $X$, which is a simplicial
set whose $0$-simplexes are prop morphisms $P\rightarrow End_{X}$
representing all the $P$-algebra structures on $X$.
More precisely, the moduli space of $P$-algebra structures on $X$
is the simplicial set such that
\[
P\{X\}=Mor_{\mathcal{P}_0}(P\otimes\Delta[-],End_X)
\]
where $(-)\otimes\Delta[-]$ is a cosimplicial resolution of $P$.
This space is a Kan complex which is homotopy invariant under weak equivalences of cofibrant props at the source
(it follows from general arguments on simplicial mapping spaces in model categories, see Chapter 16 of \cite{Hir1}).
Moreover, its connected components are exactly the homotopy classes of $P$-algebra structures on $X$.

As a consequence of Theorem 0.1, one can follow arguments
similar to those of Rezk in \cite{Rez} to characterize such a moduli space as a homotopy fiber of a map between classifying spaces.
To be explicit, we have a homotopy pullback of simplicial sets
\[
\xymatrix{P\{X\}\ar[d]\ar[r] & \mathcal{N}(fw(\mathcal{E}^c)^P)\ar[d]\\
\{X\}\ar[r] & \mathcal{N}(w\mathcal{E}^c)
}
\]
where $fw(\mathcal{E}^c)^P$ is the subcategory
of morphisms of $P$-algebras being acyclic fibrations in $\mathcal{E}$.

We will prove in a follow-up paper that the classifying space of acyclic fibrations $\mathcal{N}(fw(\mathcal{E}^c)^P)$
is actually weakly equivalent to the whole classifying space $\mathcal{N}(w(\mathcal{E}^c)^P)$.
We accordingly have a homotopy pullback which extends to the setting of algebras over dg (colored) props
the result obtained by Rezk (Theorem 1.1.5 of \cite{Rez}) in the operadic case
\[
\xymatrix{P\{X\}\ar[d]\ar[r] & \mathcal{N}(w(\mathcal{E}^c)^P)\ar[d]\\
\{X\}\ar[r] & \mathcal{N}(w\mathcal{E}^c)
}
\]

This result implies in particular that the moduli space admits a decomposition in classifying spaces of homotopy automorphisms
\[
P\{X\}\sim \coprod_{[X]} \overline{W}Lw(\mathcal{E}^c)^P(X,X)
\]
where $[X]$ ranges over the weak equivalence classes of $P$-algebras having $X$ as underlying object.
An interesting set-theoretic consequence is that the homotopy automorphisms of $P$-algebras $Lw(\mathcal{E}^c)^P(X,X)$ are homotopically small
in the sense of Dwyer-Kan \cite{DK}.

\bigskip{}

\end{document}